\newcommand{\powser}[1]{[\![#1]\!]}
\newcommand{\pdiv}{$p$-divisible }
\newcommand{\G}{\mathbb{G}}

\newcommand{\Q}{\mathbb{Q}}
\newcommand{\Z}{\mathbb{Z}}
\newcommand{\QZ}{\Q_p/\Z_p}  
\newcommand{\Zp}[1]{\Z/p^{#1}}

\newcommand{\al}{\alpha}
\newcommand{\Lk}{\Lambda_k}

\newcommand{\lra}[1]{\overset{#1}{\longrightarrow}}

\newcommand{\Prod}[1]{\underset{#1}{\prod}}
\newcommand{\Coprod}[1]{\underset{#1}{\coprod}}
\newcommand{\Colim}[1]{\underset{#1}{\colim}}

\newcommand{\E}{E_{n}}
\newcommand{\LE}{L_{K(t)}E_n}

\title{Subgroups of \pdiv groups and centralizers in symmetric groups}
\author{
Nathaniel Stapleton \\   
}
\date{\today}
\documentclass[11pt]{article}
\usepackage{amssymb,amsfonts,amsthm,amsmath,verbatim}
\usepackage[all]{xy}
\theoremstyle{definition} 

\DeclareMathOperator{\Tr}{Tr}
\DeclareMathOperator{\im}{im}
\DeclareMathOperator{\Fix}{Fix}

\DeclareMathOperator{\colim}{colim}

\DeclareMathOperator{\Spec}{Spec}
\DeclareMathOperator{\Spf}{Spf}

\DeclareMathOperator{\sub}{sub}
\DeclareMathOperator{\level}{level}

\newtheorem{thm}[subsection]{Theorem}

\newtheorem{example}[subsection]{Example}
\newtheorem{prop}[subsection]{Proposition}
\newtheorem{cor}[subsection]{Corollary}
\newtheorem{lemma}[subsection]{Lemma}

\newtheorem*{mainthm}{Theorem}

\newtheorem{remark}[subsection]{Remark}

\begin{document}
\maketitle

\section{Introduction}
There is a deep correspondence between the Morava $E$-theory of spaces and the algebraic geometry of the formal group associated to $\E$. This is apparent in theorems such as Strickland's work \cite{etheorysym} that relates the $E$-theory of symmetric groups (modulo a transfer ideal) to the scheme classifying finite subgroup schemes of the formal group. It is also seen in the work of Behrens and Rezk \cite{BehrensRezk} that provides an interpretation of the $E$-theory of the Steinberg summands $L(k)_q$ in terms of the modular isogeny complex of the formal group and in the work of Ando \cite{Isogenies} relating isogenies of the formal group to power operations in $\E$. 

The character map of Hopkins, Kuhn, and Ravenel \cite{hkr} provides a tool for understanding the Morava $E$-theory of finite groups. Not only is this map computationally useful, but it suggests a very close relationship between the chromatic filtration and the inertia groupoid functor (they call this $\Fix(-)$). This relationship has been investigated by the author in \cite{tgcm} and \cite{ttcm} in which generalizations of the character map were constructed using the algebraic geometry of \pdiv groups. 

In this paper we compute the effect of the transchromatic generalized characters of \cite{tgcm} on the Morava $E$-theory of symmetric groups. In order to provide an algebro-geometric description of the answer we must develop the relationship between transfers in $\E$ and transfers for the cohomology theory $C_t$ constructed in \cite{tgcm}. This is a generalization to higher heights of Theorem D in \cite{hkr}, which provides a straightforward formula relating transfer maps and the generalized character map. 

The computation indicates a close relationship between the cohomology of centralizers of tuples of commuting elements in symmetric groups and connected components of the scheme that classifies subgroup schemes of a particular \pdiv group. In particular, it provides algebro-geometric descriptions of the cohomology of a large class of finite groups that were without interpretation before. Although the computations in this paper make use of the cohomology theory $C_t$ and the transchromatic generalized character maps of \cite{tgcm}, we believe that they indicate more general phenomena.

To be more precise we need some setup. Fix a prime $p$. Let $\G_{\E}$ be the formal group associated to Morava $\E$. We will view this as the \pdiv group 
\[
\G_{\E}[p] \lra{} \G_{\E}[p^2] \lra{} \ldots 
\]
over $\Spec(\E^0)$. Let $0 \leq t < n$ and let $K(t)$ be Morava $K$-theory of height $t$. In \cite{tgcm}, we construct the universal $\LE^0$-algebra $C_t$ equipped with an isomorphism
\[
C_t \otimes \G_{\E} \cong (C_t \otimes \G_{\LE}) \oplus \QZ^{n-t}.
\]
Let $\G_{C_t} = C_t \otimes \G_{\LE}$. Let $X$ be a finite $G$-CW complex and let
\[
\hom(\Z_{p}^{n-t},G)
\]
be the set of continuous maps from $\Z_{p}^{n-t}$ to $G$. This is a $G$-set by conjugation and we will write 
\[
\hom(\Z_{p}^{n-t},G)/\sim
\]
for the quotient by the $G$-action.

The transchromatic generalized character map of \cite{tgcm} is a map of cohomology theories
\[
\Phi_{G}^{t}:\E^*(EG\times_G X) \lra{} C_{t}^*(EG \times_G \Fix_{n-t}^{G}(X)),
\]
where
\[
\Fix_{n-t}^{G}(X) = \Coprod{\al \in \hom(\Z_{p}^{n-t},G)}X^{\im \al}
\]
and 
\[
C_{t}^{*}(X) := C_t \otimes_{\LE^0} \LE^*(X).
\]
Because of the equivalence
\[
EG\times_G \Fix^{G}_{n-t}(X) \simeq \Coprod{[\al] \in \hom(\Z_{p}^{n-t},G)/\sim} EC_G(\im \al) \times_{C_G(\im \al)} X^{\im \al},
\]
the character map can be viewed as landing in the product of rings
\[
\Phi_{G}^{t}:\E^*(EG\times_G X) \lra{} \Prod{[\al] \in \hom(\Z_{p}^{n-t},G)/\sim}C_{t}^{*}(EC_G(\im \al) \times_{C_G(\im \al)} X^{\im \al}),
\]
where $C_G(\im \al)$ is the centralizer in $G$ of the image of $\al$.
We define
\[
\Phi_{G}^{t}[\al]:\E^*(EG\times_G X) \lra{} C_{t}^*(EC_G(\im \al) \times_{C_G(\im \al)} X^{\im \al})
\]
to be $\Phi_{G}^{t}$ composed with projection onto the factor of $[\al]$.

For $H \subseteq G$ and a cohomology theory $E$, there is a transfer map
\[
E^*(EH\times_H X) \lra{\Tr_E} E^*(EG \times_G X).
\]
\begin{mainthm} 
Let $H \subseteq G$ and $X$ be a finite $G$-space. Let $\Phi_{G}^{t}$ and $\Phi_{H}^{t}$ be the transchromatic generalized character maps associated to the groups $H$ and $G$. Then for $x \in \E^*(EH\times_H X)$ there is an equality
\[
\Phi^{t}_{G}[\al](\Tr_{\E}(x)) = \sum_{[gH] \in (G/H)^{\im \al}/C_G(\im \al)} \Tr_{C_t}(\Phi^{t}_{H}[g^{-1}\al g](x)).
\]
\end{mainthm}

When $t = 0$ this recovers Theorem D of \cite{hkr}. When $X = *$ the transfer on the right is along the inclusion
\[
gC_H(g^{-1} \im \al g)g^{-1} \subset C_G(\im \al).
\]

Let $\Sigma_{p^{k-1}}^{\times p} \subseteq \Sigma_{p^k}$ be the obvious subgroup. Let $I_{tr} \subseteq \E^0(B\Sigma_{p^k})$ be the ideal generated by the image of the transfer along $\Sigma_{p^{k-1}}^{\times p} \subseteq \Sigma_{p^k}$.

In \cite{etheorysym}, Strickland proves that
\[
\Spec(\E^0(B\Sigma_{p^k})/I_{tr}) \cong \sub_{k}(\G_{\E}),
\]
where $\sub_{k}(\G_{\E})$ is the scheme that classifies subgroup schemes of order $p^k$ in $\G_{\E}$.

The transchromatic generalized character map and the theorem above provide an isomorphism
\[
C_t \otimes_{\E^0} \E^0(B\Sigma_{p^k})/I_{tr} \cong \Prod{[\al] \in \hom(\Z_{p}^{n-t}, \Sigma_{p^k})/\sim} C_{t}^{0}(BC(\im \al))/I_{tr}^{[\al]},
\]
in which the ideals $I_{tr}^{[\al]}$ in the codomain are constructed using the theorem above. Each of the factors in the codomain are connected.

Let $m$ be the smallest integer such that $\al$ factors
\[
\xymatrix{& \Sigma_{p^{m}}^{\times p^{k-m}} \ar[d] \\ \Z_{p}^{h} \ar[r]^{\al} \ar@{-->}[ur] & \Sigma_{p^{k}}}
\]
up to conjugacy. Now let 
\[
\Delta: \Sigma_{p^{m}} \lra{} \Sigma_{p^{m}}^{\times p^{k-m}}
\]
be the diagonal map. The ideal $I_{tr}^{[\al]} \neq C_{t}^{0}(BC(\im \al))$ if and only if $\al$ further factors through $\Delta$ (up to conjugacy).

There are also isomorphisms
\[
C_t \otimes \sub_{k}(\G_{\E}) \cong \sub_{k}(C_t \otimes \G_{\E}) \cong \sub_{k}(\G_{C_t} \oplus \QZ^{n-t}).
\]
The main theorem of the paper describes the relationship between the composite of these isomorphisms and the character map:
\begin{mainthm}
The isomorphism fits into a commutative triangle
\[
\xymatrix{\Coprod{[\al] \in \hom(\Z_{p}^{n-t}, \Sigma_{p^k})/\sim}\Spec(C_{t}^{0}(BC(\im \al))/I_{tr}^{[\al]}) \ar[r]^-{\cong} \ar[d] & \sub_{k}(\G_{C_t} \oplus \QZ^{n-t}) \ar[dl] \\ \sub_{\leq k}(\QZ^{n-t}), &}
\]
where the left map takes the component corresponding to $[\al]$ to the image of the Pontryagin dual
\[
\al^*: (\im \al)^* \lra{} \QZ^{n-t}
\]
and the right map is induced by the projection
\[
\G_{C_t} \oplus \QZ^{n-t} \lra{} \QZ^{n-t}.
\]
\end{mainthm}

This implies the following: Fix a map $\al: \Z_{p}^{n-t} \lra{} \Sigma_{p^k}$ that factors through $\Delta$ (up to conjugacy) and let $L \subseteq \QZ^{n-t}$ be the image of the Pontryagin dual $\al^*: \im \al \lra{} \QZ^{n-t}$. Let $f:\sub_{k}(\G_{C_t} \oplus \QZ^{n-t}) \lra{} \sub_{\leq k}(\QZ^{n-t})$ and let $f^{-1}(L)$ be the pullback 
\[
\xymatrix{f^{-1}(L) \ar[r] \ar[d] & \sub_{k}(\G_{C_t} \oplus \QZ^{n-t}) \ar[d]^{f} \\
			\ast \ar[r]^(.3){L} & \sub_{\leq k}(\QZ^{n-t}).}
\]
Then
\[
\Spec(C_{t}^{0}(BC(\im \al))/I_{tr}^{[\al]}) \cong f^{-1}(L).
\]
That is, the pullback gives an algebro-geometric interpretation of 
\[
\Spec(C_{t}^{0}(BC(\im \al))/I_{tr}^{[\al]}):
\]
it consists of the subgroups of order $p^k$ in $\G_{C_t} \oplus \QZ^{n-t}$ that project onto $L$ in $\QZ^{n-t}$.

\paragraph{Acknowledgments} It is a pleasure to thank Mark Behrens for suggesting that the transchromatic generalized character maps be applied to the symmetric groups and for many helpful conversations. I would like to thank Charles Rezk for pointing out Theorem D in \cite{hkr} and for conversations regarding subgroups of \pdiv groups. I would also like to thank Matt Ando, Tobi Barthel, Dustin Clausen, David Gepner, Mike Hopkins, Zhen Huan, Jacob Lurie, Haynes Miller, Kyle Ormsby, and Vesna Stojanoska for helpful conversations. 

\section{Transfer Maps}
We recall the formula provided in Theorem D of \cite{hkr} and construct a generalization. Following Adams advice at the end of Chapter 4 of \cite{infiniteloop}, we avoid mentioning the words ``double cosets''.

\subsection{Recollections}
Fix a prime $p$ and an integer $0\leq t<n$. Let $\E$ be Morava $E$-theory of height $n$ with associated formal group $\G_{\E}$. Let $\LE$ be the localization of $\E$ with respect to height $t$ Morava $K$-theory $K(t)$. Recall from Section 3 of \cite{tgcm} that 
\[
C_t' = \Colim{k} \text{ } \LE^0\otimes_{\E^0}\E^0(B(\Zp{k})^{n-t})
\]
and 
\[
C_t = S^{-1}C_t',
\]
where $S$ is essentially the image of $\QZ^{n-t}$ inside of $\G_{\E}(C_t')$. By Corollary 2.18 of \cite{tgcm} the ring $C_t$ is the universal $\LE^0$-algebra equipped with an isomorphism
\[
C_t \otimes (\LE^0 \otimes \G_{\E}) \cong (C_t \otimes \G_{\LE}) \oplus \QZ^{n-t}.
\]

Parting from the notation in \cite{tgcm}, we will often write $\G_{C_t}$ for $C_{t}\otimes \G_{\LE}$. Recall that for a finite $G$-space (a space equivalent to a finite $G$-CW complex) $X$,
\[
\Fix_{n-t}^{G}(X) = \Coprod{\al \in \hom(\Z_{p}^{n-t},G)}X^{\im \al}.
\]
The main construction of \cite{tgcm} is the transchromatic generalized character map
\[
\Phi_{G}^{t}:\E^*(EG\times_G X) \lra{} C_{t}\otimes_{\LE^0}\LE^*(EG \times_G \Fix_{n-t}^{G}(X)).
\]
It recovers the generalized character map of \cite{hkr} when $t=0$. We will denote the codomain as 
\[
C_{t}^{*}(EG \times_G \Fix_{n-t}^{G}(X)).
\]
Recall that the character map $\Phi_{G}^{t}$ is the composite of two maps. The first is induced by a map of topological spaces
\[
B\Lk \times EG \times_G \Fix^{G}_{n-t}(X) \lra{T} EG\times_G X,
\]
which is induced by a map of topological groupoids
\[
\xymatrix{\Lk \times G \times \Fix^{G}_{n-t}(X) \ar[d] \ar@<1ex>[d] \ar[r]^-{T_{\text{mor}}} & G \times X \ar[d] \ar@<1ex>[d] \\ \Fix^{G}_{n-t}(X) \ar[r]^-{T_{\text{ob}}} & X.}
\]
The map $T_{\text{ob}}$ is just the inclusion on each component. The map $T_{\text{mor}}$ is defined by
\[
(l,g,x \in X^{\im \al}) \mapsto (g \al(l), x).
\]
More details can be found in Section 3.1 of \cite{tgcm} or Section 4.1 of \cite{ttcm}.

In \cite{hkr}, Hopkins, Kuhn, and Ravenel provide a formula for the relationship between transfers for Morava $\E$ and their character map (the case $t=0$ above).
\begin{thm} (\cite{hkr}, Theorem D)
Let $X$ be a finite $G$-space, $H \subseteq G$, $x \in \E^0(EH \times_H X)$, $\al:\Z_{p}^{n-t} \lra{} G$, and 
\[
\Tr_{\E}: \E^*(EH\times_H X) \lra{} \E^*(EG \times_G X)
\]
the transfer map in $\E$. Let $\Phi^{0}_{G}$ be the Hopkins-Kuhn-Ravenel character map of \cite{hkr} and $\Phi^{0}_{G}(\al)$ the character map followed by the projection onto the $\al$-factor. Then
\[
\Phi^{0}_{G}(\al)(\Tr(x)) = \sum_{gH \in (G/H)^{\im \al}} \Phi^{0}_{H}(g^{-1}\al g)(x).
\]
\end{thm}

We will extend their proof methods in order to generalize their result to $t > 0$. 

\subsection{Two pullback squares} \label{twopb}
Fix a finite group $G$, a subgroup $H$, and an integer $k$ such that every continuous map $\Z_{p}^{n-t} \lra{} G$ factors through $\Lk = (\Zp{k})^{n-t}$. 

\begin{lemma} \label{prop1}
For $\al:\Lk \lra{} G$, let $gH \in (G/H)^{\im \al} \subseteq \Fix^{G}_{n-t}(G/H)$. Then $\im g^{-1} \al g \subseteq H$.
\end{lemma}
\begin{proof}
Let $a \in \im \al$. Then $agH = gH$ implies that $g^{-1}agH = H$. Now $g^{-1}ag$ fixes $H$ implies that $g^{-1}ag \in H$.
\end{proof}

For $\al:\Lk \lra{} G$, let $C(\im \al)$ be the centralizer of the image of $\al$. When multiple groups are in use we may write $C_H(\im \al)$ to mean the centralizer of $\im \al$ inside of $H$. 
Let $X$ be a finite $G$-space. Recall that $X^{\im \al}$ is a $C_G(\im \al)$-space. There is an equivalence of spaces 
\[
EH\times_H X \simeq EG \times_G (G\times_H X),
\]
where $G\times_H X$ is the obvious coequalizer.
Recall that there is a homeomorphism of $G$-spaces
\[
G \times_H X \cong G/H \times X
\]
induced by the map
\[
(g,x) \mapsto (gH, gx).
\]
Fix a map $\al:\Lk \lra{} G$. The above homeomorphism induces a homeomorphism of $C(\im \al)$-spaces
\[
(G\times_H X)^{\im \al} \cong (G/H \times X)^{\im \al} \cong (G/H)^{\im \al}\times X^{\im \al}. 
\]

There is also an equivalence of spaces
\begin{equation*} \label{eqn}
w:\Coprod{[\al] \in \hom(\Z_{p}^{n-t},G)/\sim} EC(\im \al) \times_{C(\im \al)} X^{\im \al} \simeq EG\times_G \Fix^{G}_{n-t}(X),
\end{equation*}
where the disjoint union is taken over conjugacy classes of maps. The description on the left is given by fixing representatives of conjugacy classes. This equivalence follows from Proposition 4.13 in \cite{ttcm}. Given a representative $\al \in [\al]$, the map is induced by the inclusion $C(\im \al) \subseteq G$.

\begin{prop} \label{prop2}
There is a pullback of spaces
\[
\xymatrix{B\Lk \times EH \times_H \Fix_{n-t}^{H}(X) \ar[d] \ar[r]^-{T} & EH\times_H X \ar[d] \\
			 B\Lk \times EG \times_G \Fix_{n-t}^{G}(G/H\times X) \ar[r]^-{T} & EG \times_G (G/H \times X). }
\]
\end{prop}
\begin{proof}
Begin by viewing the spaces as the realizations of topological groupoids. The right hand map is induced by $x \mapsto (eH,x)$. The diagram of topological groupoids is a pullback. It is trivial to see this on the level of objects. The bottom arrow on morphisms is
\[
(l,g,(gH,x) \in (G/H\times X)^{\im \al}) \mapsto (g\al(l), (gH,x)).
\]
The image of this is only hit by $(h,x)$ if $\im \al \subseteq H$ and $g \in H$ in which case it is hit by $(g\al(l),x)$. This completes the proof as realization commutes with pullbacks (see Chapter 11 of \cite{loopspaces}).

 \end{proof}

\begin{cor} \label{cor}
There is a homotopy commutative diagram
\[
\xymatrix{B\Lk \times \Coprod{[\beta] \in \hom(\Z_{p}^{n-t},H)/\sim} EC_H(\im \beta) \times_{C_H(\im \beta)} X^{\im \beta} \ar[d] \ar[r] & EH\times_H X \ar[d] \\
			 B\Lk \times \Coprod{[\al] \in \hom(\Z_{p}^{n-t},G)/\sim} EC_G(\im \al) \times_{C_G(\im \al)} (G/H\times X)^{\im \al} \ar[r] & EG \times_G (G/H \times X). }
\]
\end{cor}
\begin{proof}
This follows immediately from the previous proposition and the equivalence $w$.
\end{proof}

Note that the right map is an equivalence. In the next section we will spend a significant amount of space analyzing the left map. We will show that it is an equivalence and give a formula for the map.

\begin{prop} \label{prop3}
There is a pullback of spaces
\[
\xymatrix{\Coprod{[\al] \in \hom(\Z_{p}^{n-t},G)/\sim}E(\Lk \times C(\im \al)) \times_{\Lk \times C(\im \al)} (G/H \times X^{\im \al}) \ar[r] \ar[d] & EG \times_G (G/H \times X) \ar[d] \\ \Coprod{[\al] \in \hom(\Z_{p}^{n-t},G)/\sim} B \Lk \times EC(\im \al) \times_{C(\im \al)} X^{\im \al} \ar[r]^-{T \circ (B\Lk \times w)} & EG \times_G X,}
\]
where the map on the right is induced by the projection and the map on the bottom is the topological part of the character map.
\end{prop}
\begin{proof}
Once again, viewing the spaces as the realization of topological groupoids makes this easy to see. It is clearly a pullback on the level of spaces of objects and spaces of morphisms. It is important to note that $C(\im \al)$ acts diagonally on $G/H \times X^{\im \al}$ and that $\Lk$ need not act trivially on the elements of $G/H$. This is why $B\Lk$ does not split off as a factor in the pullback.
\end{proof}

Following the proof of Theorem D in \cite{hkr}, consider the decomposition of $\Lk \times C(\im \al)$ spaces
\[
G/H \times X^{\im \al} \cong ((G/H)^{\im \al} \times X^{\im \al}) \coprod ((G/H)^{\im \al} \times X^{\im \al})^c,
\]
where $(-)^c$ denotes the complement. This splits $G/H \times X^{\im \al}$ into the part fixed by the action of $\Lk$ through $\al$ and the part that is not fixed.

Note that we can use this to decompose the pullback
\[
\Coprod{[\al] \in \hom(\Z_{p}^{n-t},G)/\sim}E(\Lk \times C(\im \al)) \times_{\Lk \times C(\im \al)} (G/H \times X^{\im \al})
\]
as the disjoint union of
\[
B\Lk \times \Coprod{[\al] \in \hom(\Z_{p}^{n-t},G)/\sim} EC(\im \al) \times_{C(\im \al)} (G/H\times X)^{\im \al}
\]
and 
\[
\Coprod{[\al] \in \hom(\Z_{p}^{n-t},G)/\sim}E(\Lk \times C(\im \al)) \times_{\Lk \times C(\im \al)} (G/H^{\im \al} \times X^{\im \al})^c.
\]
Also note that when the top map in Proposition \ref{prop3} is restricted to
\[
B\Lk \times \Coprod{[\al] \in \hom(\Z_{p}^{n-t},G)/\sim} EC(\im \al) \times_{C(\im \al)} (G/H\times X)^{\im \al}
\]
then it is just $T \circ w$ for the $G$-space $G/H\times X$. 

\subsection{Some computations}
For applications it is useful to be able to explicitly compute the left vertical map of Corollary \ref{cor}.

Let $i:H \hookrightarrow G$ be the inclusion. Let 
\[
\hom(\Z_{p}^{n-t},G)/\sim
\]
be the set of conjugacy classes of map from $\Z_{p}^{n-t}$ to $G$ under conjugation by $G$.

Consider the map
\[
i_*:\hom(\Z_{p}^{n-t},H)/\sim \text{} \lra{} \hom(\Z_{p}^{n-t},G)/\sim
\]
induced by $i$. 
Then 
\[
i_{*}^{-1}([\al]) = \{[\beta] \in \hom(\Z_{p}^{n-t},H)/\sim | [i \circ \beta] = [\al] \in \hom(\Z_{p}^{n-t},G)/\sim \}.
\]
\begin{prop} \label{proponeandahalf}
There is a bijection
\[
(G/H)^{\im \al}/C(\im \al) \cong i_{*}^{-1}([\al]).
\] 
\end{prop}
\begin{proof}
Let $gH \in (G/H)^{\im \al}$. Send $gH$ to $[g^{-1}\al g]$. Since $gH$ is fixed by $\im \al$, Lemma \ref{prop1} implies that $g^{-1} \al g \subseteq H$. Let $kH \in (G/H)^{\im \al}$ with $kH \neq gH$. If $kH = cgH$ for $c \in C(\im \al)$ then there exists $h \in H$ such that
\[
kh = cg
\] 
and
\[
[h^{-1}k^{-1} \al k h] = [k^{-1}\al k] = [g^{-1}c^{-1} \al cg] = [g^{-1} \al g]
\]
in $\hom(\Z_{p}^{n-t},H)/\sim$.
However, if $kH \neq cgH$ for some $c \in C(\im \al)$ then
\[
[k^{-1} \al k] \neq [g^{-1} \al g] \text{ in } \hom(\Z_{p}^{n-t},H)/\sim
\]
but
\[
[g^{-1}kk^{-1} \al k k^{-1}g] = [g^{-1} \al g] \in \hom(\Z_{p}^{n-t},G)/\sim.
\]
\end{proof}

Fix an $[\al] \in \hom(\Z_{p}^{n-t},G)/\sim$. The homotopy equivalence $w$ of Section 2.3 restricted to the component of $[\al]$ gives the homotopy equivalence
\[
w_{[\al]}:EC(\im \al) \times_{C(\im \al)} (G/H)^{\im \al}\times X^{\im \al} \lra{\simeq} EG\times_G \Coprod{\gamma \in [\al]} (G/H)^{\im \gamma} \times X^{\im \gamma}.
\]
We analyze the inverse equivalences.
\begin{prop}
Let $g_1,\ldots,g_h$ be elements of $G$ such that 
\[
\{g_{1}\al g_{1}^{-1}, \ldots , g_{h} \al g_{h}^{-1}\} = [\al].
\]
Then $g_1,\ldots,g_h$ determine an inverse equivalence to $w_{[\al]}$.
\end{prop}
\begin{proof}
We write down the inverse equivalence in terms of the associated topological groupoids. On objects we send
\[
(gH,x) \in (G/H\times X)^{\im g_i \al g_{i}^{-1}} \mapsto (g_{i}^{-1}gH,g_{i}^{-1}x) \in (G/H \times X)^{\im \al}.
\]
The map on morphisms is a bit more complicated. We construct it by using what it needs to do on objects. Recall that $k \in G$ acts on $(G/H \times X)^{\im \gamma}$ by sending
\[
k:(gH,x) \mapsto (kgH,kx) \in (G/H \times X)^{\im k\gamma k^{-1}}.
\]
We may assume that $kg_i \al (kg_i)^{-1} = g_j \al g_{j}^{-1}$ where $g_j \in \{g_1, \ldots , g_h\}$. In order to determine where the morphism $(k,gH,x) \in G \times (G/H \times X)^{\im g_i \al g_{i}^{-1}}$ must map to in $C(\im \al) \times (G/H \times X)^{\im \al}$ consider the following diagram
\[
\xymatrix{(gH,x) \ar[r]^k \ar[d]^{g_{i}^{-1}} & (kgH,kx) \ar[d]^{g_{j}^{-1}} \\ (g_{i}^{-1} gH, g_{i}^{-1}x) & (g_{j}^{-1}kgH,g_{j}^{-1}kx).}
\]
The composite of the horizontal and then right map is the target composed with the map on objects. The left map is the source (projection) and then the map on objects.
We see from this diagram that we must map
\[
(k,gH,x) \in G \times (G/H \times X)^{\im g_i \al g_{i}^{-1}} \mapsto (g_{j}^{-1}kg_i, g_{i}^{-1}gH, g_{i}^{-1}x) \in C(\im \al) \times (G/H \times X)^{\im \al}.
\]
We check that $g_{j}^{-1}kg_i \in C(\im \al)$: Let $a \in \im \al$ then
\[
g_{j}^{-1}kg_ia(g_{j}^{-1}kg_i)^{-1} = g_{j}^{-1}kg_iag_{i}^{-1}k^{-1}g_j = g_{j}^{-1}g_jag_{j}^{-1}g_j = a.
\]
It is not hard (but takes a lot of space) to show that this is in fact an inverse equivalence. 

\end{proof}
We can now provide a formula for the left map in Corollary \ref{cor}:
\[
\xymatrix{\Coprod{[\beta] \in \hom(\Z_{p}^{n-t},H)/\sim} EC_H(\im \beta) \times_{C_H(\im \beta)} X^{\im \beta} \ar[d] \\ \Coprod{[\al] \in \hom(\Z_{p}^{n-t},G)/\sim} EC_G(\im \al) \times_{C_G(\im \al)} (G/H\times X)^{\im \al}.}
\]
We do this by tracing through the diagram
\[
\xymatrix{\Coprod{[\beta] \in \hom(\Z_{p}^{n-t},H)/\sim} EC_H(\im \beta) \times_{C_H(\im \beta)} X^{\im \beta} \ar[r]^-{w} \ar[d] & EH \times_H \Fix_{n-t}^{H}(X) \ar[d] \\
			\Coprod{[\al] \in \hom(\Z_{p}^{n-t},G)/\sim} EC_G(\im \al) \times_{C_G(\im \al)} (G/H\times X)^{\im \al} \ar[r]^-{w} & EG \times_G \Fix_{n-t}^{G}(X)}
\]
using the inverse equivalence described in the previous proposition. 

Fix an $[\al] \in \hom(\Z_{p}^{n-t},G)/\sim$ such that $(G/H)^{\im \al} \neq \emptyset$. Let $g_1,\ldots,g_h$ be elements of $G$ such that 
\[
\{g_{1}\al g_{1}^{-1}, \ldots , g_{h} \al g_{h}^{-1}\} = [\al]
\]
as in the previous proposition. Let $l$ be the cardinality $i^*([\al])$. Without loss of generality, let 
\[
\beta_i := g_i^{-1} \al g_i, \text{ where } i\in {1\ldots l}
\]
be representatives for the elements of $i^*([\al])$. Even more, to simplify the formulas, let us take these representatives to be the chosen ones in the top left corner. 

By using the topological groupoid model for these spaces, we compute the map
\[
\Coprod{\{[\beta_1],\ldots, [\beta_l]\}}EC_H(\im \beta_i)\times_{C_H(\im \beta_i)} X^{\im \beta_i} \lra{} EC_G(\im \al) \times_{C_G(\im \al)} (G/H)^{\im \al} \times X^{\im \al}.
\]
Let $(c,x) \in C_H(\im \beta_i) \times X^{\im \beta_i}$ then we have (on morphism sets)
\[
\xymatrix{(c,x) \in C_H(\im \beta_i) \times X^{\im \beta_i} \ar@{|->}[r] & (c,x) \in H \times X^{\im \beta_i} \ar@{|->}[d] \\ (g_{i}cg_{i}^{-1},g_iH,g_ix) \in C_G(\im \al) \times (G/H\times X)^{\im \al} & (c,eH,x) \in G \times (G/H\times X)^{\im \beta_i} \ar@{|->}[l]}
\] 
We are using the fact that $c \in C_H(\im \beta_i)$ to compute the bottom arrow. 

To show that the map is an equivalence we will show that it is essentially surjective and fully faithful. Essential surjectivity follows easily from Lemma \ref{prop1}. Thus it suffices to show that the map induces an isomorphism on automorphism groups. Let $(gH,gx) \in (G/H \times X)^{\im \al}$ be hit by $x \in X^{\im g^{-1} \al g}$ under the map defined above. Consider the stabilizers $\text{Stab}(gH,gx) \subseteq C_G(\im \al)$ and $\text{Stab}(x) \subseteq C_H(\im g^{-1} \al g)$. These map to each other by conjugation by $g$. This is clearly injective. We show that conjugation by $g^{-1}$ produces an isomorphism. Consider $c \in \text{Stab}(gH,gx) \subseteq C_G(\im \al)$. We have that
\[
cgx = gx
\]
and thus $g^{-1}cg$ stabilizes $x$. This is not enough though; we must show that $g^{-1}cg \in C_H(\im g^{-1} \al g)$. Clearly $g^{-1}cg$ centralizes $\im g^{-1} \al g$ and also
\[
cgH = gH
\]
implies that $g^{-1}cg \in H$. We have proved the following: 
\begin{prop}\label{equivprop}
Fix an $[\al] \in \hom(\Z_{p}^{n-t},G)/\sim$ such that $(G/H)^{\im \al} \neq \emptyset$. Let $g_1,\ldots,g_h$ be elements of $G$ such that 
\[
\{g_{1}\al g_{1}^{-1}, \ldots , g_{h} \al g_{h}^{-1}\} = [\al].
\]
This determines an equivalence
\[
\Coprod{[\beta] \in i_{*}^{-1}[\al]} EC_H(\im \beta) \times_{C_H(\im \beta)} X^{\im \beta} \simeq EC_G(\im \al) \times_{C_G(\im \al)} (G/H\times X)^{\im \al}.
\]
\end{prop}

\begin{remark}One of the main things to take away from this discussion is the following: Consider $(G/H)^{\im \al}$ with the action by $C_G(\im \al)$. Let $gH \in (G/H)^{\im \al}$, then the stabilizer of $gH$ is precisely $gC_H(g^{-1} \im \al g)g^{-1}$. 

It is important to note that, even if $\im \al \subseteq H$ and $gH \in (G/H)^{\im \al}$, the inclusion
\[
C_H(\im g^{-1} \al g) \subseteq g^{-1}C_H(\im \al)g
\]
need not be an equality because $g$ is not necessarily in $H$. 
\end{remark}
\subsection{Properties of transfers}
Taking our cue from Section 6.5 of \cite{hkr} (who follow \cite{infiniteloop}, Chapter 4), we consider the following properties of the transfer map associated to a finite covering of spaces $W \lra{} Z$ for a cohomology theory $E$:

\begin{enumerate}
\item the transfer associated to the identity map is the identity map;
\item if $W_1 \coprod W_2 \lra{} Z$ is a disjoint union of finite coverings, then the transfer map 
\[
E^*(W_1) \oplus E^*(W_2) \lra{} E^*(Z)
\]
is the sum of the transfer maps associated to the coverings $W_1 \lra{} Z$ and $W_2 \lra{} Z$;
\item the transfer $E^*(W) \lra{} E^*(Z)$ is a map of $E^*(Z)$-modules;
\item if
\[
\xymatrix{W_1 \ar[r] \ar[d] & W \ar[d] \\ Z_1 \ar[r] & Z}
\]
is a fiber square, then the diagram
\[
\xymatrix{E^*(W_1) \ar[d]^{\text{Tr}} & E^*(W) \ar[l] \ar[d]^{\text{Tr}} \\ E^*(Z_1) & E^*(Z) \ar[l]}
\]
commutes.
\end{enumerate}

We also need direct analogues of Lemma 6.12 and Corollary 6.13 of \cite{hkr}.
\begin{prop}
If $A \subset \Lk$ is a proper subgroup, then the composite
\[
\E^*(BA) \lra{\text{Tr}} \E^*(B\Lk) \lra{} C_{t}^*
\]
is zero.
\end{prop}
\begin{proof}
The construction of $C_{t}^*$ parallels the construction of $C_{0}^* = L(\E^*)$ from \cite{hkr}. Their proof goes through.
\end{proof}

\begin{cor} \label{emptycor}
Suppose that $Y$ is a trivial $\Lk$-space, and that $J$ is a finite $\Lk$-set with 
\[
J^{\Lk} = \emptyset.
\]
Then the composite
\[
\E^*(E\Lk \times_{\Lk} (J\times Y)) \lra{\Tr} \E^*(B\Lk \times Y) \lra{} C_{t}^* \otimes_{\LE^*} \LE^*(Y)
\]
is zero.
\end{cor}
\begin{proof}
This follows immediately from the previous proposition and the proof of Corollary 6.13 in \cite{hkr}.
\end{proof}

The following will be useful for later computations.
\begin{prop} \label{augmentation}
Let $t > 0$. Assume that $H \subset G$ is a subgroup and that $p$ divides the order of $G/H$. Let $I_{\text{aug}}$ be the kernel of the map 
\[
C_{t}^{*}(BG) \lra{} C_{t}^{*}(Be),
\] 
where $e$ is the trivial subgroup. The image of the transfer
\[
C_{t}^{*}(BH) \lra{\Tr} C_{t}^{*}(BG)
\]
is contained in the ideal $(p)+I_{\text{aug}}$.
\end{prop}
\begin{proof}
The proof is an application of Properties 4 and 2 above. Consider the pullback diagram of $G$-sets
\[
\xymatrix{G \times G/H \ar[r] \ar[d] & G/H \ar[d] \\ G \ar[r] & G/G}.
\]
The group $G$ acts freely on the pullback so it is isomorphic to $\Coprod{G/H}G$. Applying Property 4 we get the commutative diagram
\[
\xymatrix{\Prod{G/H}C_{t}^{*} \ar[d]^{\Tr} & C_{t}^{*}(BH) \ar[l] \ar[d]^{\Tr} \\ C_{t}^{*} & C_{t}^{*}(BG) \ar[l].}
\]
The left arrow is just multiplication by $|G/H|$ by Property 2.
\end{proof}


\subsection{Transfers for transchromatic character maps}
We use the properties of transfer maps and the pullbacks and decompositions discussed in the previous section to provide a formula relating transfer maps for $\E$ and $C_t$ and the transchromatic generalized character maps.

Before proving the theorem we establish one bit of notation. Because of the equivalence
\[
EG\times_G \Fix^{G}_{n-t}(X) \simeq \Coprod{[\al] \in \hom(\Z_{p}^{n-t},G)/\sim} EC(\im \al) \times_{C(\im \al)} X^{\im \al},
\]
the character map can be viewed as landing in the product of rings
\[
\Phi_{G}^{t}:\E^*(EG\times_G X) \lra{} \Prod{[\al] \in \hom(\Z_{p}^{n-t},G)/\sim}C_{t}^{*}(EC(\im \al) \times_{C(\im \al)} X^{\im \al}).
\]
We define
\[
\Phi_{G}^{t}[\al]:\E^*(EG\times_G X) \lra{} C_{t}^*(EC(\im \al) \times_{C(\im \al)} X^{\im \al})
\]
to be $\Phi_{G}^{t}$ composed with projection onto the factor of $[\al]$.

\begin{thm} \label{mainthm}
Let $H \subseteq G$ and $X$ be a finite $G$-space. Let $\Phi_{G}^{t}$ and $\Phi_{H}^{t}$ be the transchromatic generalized character maps associated to the groups $H$ and $G$. Then for $x \in \E^*(EH\times_H X)$ there is an equality
\[
\Phi^{t}_{G}[\al](\Tr_{\E}(x)) = \sum_{[gH] \in (G/H)^{\im \al}/C(\im \al)} \Tr_{C_t}(\Phi^{t}_{H}[g^{-1}\al g](x)).
\]
\end{thm}
\begin{proof}
Fix an $\al:\Z_{p}^{n-t} \lra{} G$. Our goal is to analyze $\Phi_{G}^{t}[\al]$. 

We begin by applying $\E$ to the pullback diagram from Proposition \ref{prop3} specialized to $[\al]$. We get the diagram
\[
\xymatrix{\E^*(EG \times_G (G/H \times X)) \ar[r] \ar[d]^{\Tr} & \E^*(E(\Lk \times C_G(\im \al)) \times_{\Lk \times C_G(\im \al)} (G/H \times X^{\im \al})) \ar[d]^{\Tr} \\ \E^*(EG \times_G X) \ar[r] & \E^*(B \Lk \times EC_G(\im \al) \times_{C_G(\im \al)} X^{\im \al}).} 
\]
Using the decomposition noted at the end of Subsection \ref{twopb} and Corollary \ref{emptycor}, on the right hand side of square above we can restrict our attention to
\[
\xymatrix{\E^*(B\Lk \times EC_G(\im \al) \times_{C_G(\im \al)} (G/H)^{\im \al} \times X^{\im \al}) \ar[d] \\ \E^*(B \Lk \times EC_G(\im \al) \times_{C_G(\im \al)} X^{\im \al}).}
\]
Now using the square from Proposition \ref{prop2} we arrive at the commutative diagram
\[
\xymatrix@=8pt{\E^*(EH \times_H X) \ar[r] \ar[dd]^{\cong} & \E^*(B\Lk \times \Coprod{[\beta] \in i_{*}^{-1}[\al]}EC_H(\im \beta) \times_{C_H(\im \beta)} X^{\im \beta}) \ar[dd]^{\cong} \\ \\ \E^*(EG \times_G (G/H \times X)) \ar[r] \ar[dd]^{\Tr} & \E^*(B\Lk \times EC_G(\im \al) \times_{C_G(\im \al)} (G/H)^{\im \al} \times X^{\im \al}) \ar[dd]^{\Tr} \\ \\ \E^*(EG \times_G X) \ar[r] & \E^*(B \Lk \times EC_G(\im \al) \times_{C_G(\im \al)} X^{\im \al}).} 
\]
The top right isomorphism follows from Proposition \ref{equivprop}. All of the horizontal maps are portions of the topological part of the transchromatic generalized character map. Applying the algebraic part of the transchromatic generalized character map and the fact that transfers commute with maps of cohomology theories (the transfer map is just a map of spectra), we get
\[
\xymatrix{\E^*(EH \times_H X) \ar[r]^-{\prod \Phi_{H}^{t}[\beta]} \ar[d]^{\Tr} & \Prod{[\beta] \in i_{*}^{-1}[\al]}C_{t}^{*}(EC_H(\im \beta) \times_{C_H(\im \beta)} X^{\im \beta}) \ar[d]^{\sum \Tr} \\ \E^*(EG \times_G X) \ar[r]^-{\Phi_{G}^{t}[\al]} & C_{t}^{*}(EC_G(\im \al) \times_{C_G(\im \al)} X^{\im \al}).} 
\]
By Proposition \ref{proponeandahalf} the top right corner of this square can be rewritten as
\[
\Prod{[gH] \in (G/H)^{\im \al}/C_G(\im \al)}C_{t}^{*}(EC_H(\im g^{-1}\al g) \times_{C_H(\im g^{-1}\al g)} X^{\im g^{-1}\al g}).
\]
\end{proof}

\begin{cor} \label{maincor}
Let $\al: \Z_{p}^{n-t} \lra{} G$, $H \subseteq G$, and $gH \in (G/H)^{\im \al}$. When $X=*$ the transfer map in the formula can be taken to be along the inclusion
\[
gC_H(g^{-1} \im \al g)g^{-1} \subseteq C_G(\im \al).
\]
\end{cor}
\begin{proof}
This follows from the remark at the end of Subsection \ref{twopb}.
\end{proof}

\begin{remark}
This is a higher chromatic analogue for the formula for the character of an induced representation. For $H \subseteq G$, $u \in G$, and $\chi$ a class function on $H$,
\begin{align*}
\chi \uparrow_{H}^{G}(u) &= \frac{1}{|H|} \sum_{g \in G, \text{ } g^{-1}ug \in H} \chi(g^{-1}ug) \\
&= \sum_{gH \in (G/H)^u} \chi(g^{-1}ug) \\
&= \sum_{[gH] \in (G/H)^u/C(u)}[C_G(u):gC_H(g^{-1} u g)g^{-1}]\chi(g^{-1}ug).
\end{align*}
\end{remark}

\section{Decomposing the Subgroup Scheme}
We use the transfer maps constructed in the previous section to calculate how the scheme
\[
\sub_{k}(\G_{\E}) = \Spec \E^0(B\Sigma_{p^k})/I_{tr}
\]
decomposes under base change to $C_t$. We provide an algebro-geometric interpretation of the resulting decomposition.
\subsection{Recollections}
In Section 10 of \cite{subgroups}, Strickland defines a formal scheme 
\[
\sub_{k}(\G_{\E}),
\]
which represents the functor
\[
\sub_{k}(\G_{\E}):\text{complete Noetherian local $\E^0$-algs} \lra{} \text{Set}
\]
that sends
\[
R \mapsto \{\text{subgroup schemes of order $p^k$ of } R\otimes \G_{\E}\}.
\]
The main algebro-geometric result that we need regarding $\sub_{k}(\G_{\E})$ is Theorem 10.1 of \cite{subgroups}.
\begin{thm} (\cite{subgroups}, Theorem 10.1)
For any continuous map $\E^0 \lra{} S$,
\[
S \otimes \sub_k(\G_{\E}) \cong \sub_k(S \otimes \G_{\E}).
\]
The projection $\sub_k(\G_{\E}) \lra{} \Spf(\E^0)$ is a finite free map of degree
\[
d = \text{number of subgroups of } \QZ^{n} \text{ of order } p^k.
\]
The scheme $\sub_k(\G_{\E})$ is Gorenstein.
\end{thm}
Note that Strickland's results are more general because they apply to an arbitrary formal group $\G$. Here we have presented his theorem specialized to $\G_{\E}$, the formal group associated to Morava $\E$. We will not use that the scheme is Gorenstein here.

Following Strickland, we call subgroups of $\Sigma_{p^k}$ of the form $\Sigma_i \times \Sigma_j$ with $i,j>0$ proper partition subgroups. Let $I_{tr}$ be the ideal of $\E^0(B\Sigma_{p^k})$ generated by the images of the transfers of the proper partition subgroups. In \cite{etheorysym}, Strickland proves the main topological result regarding $\sub_{k}(\G_{\E})$.
\begin{thm} \label{strickland} (\cite{etheorysym}, Proposition 9.1)
There is an isomorphism
\[
\Spf (\E^0(B\Sigma_{p^k})/I_{tr}) \cong \sub_{k}(\G_{\E}).
\]
\end{thm}
Lemma 8.11 of \cite{etheorysym} implies that we need only consider the ideal generated by the image of the transfer from $\Sigma_{p^{k-1}}^{\times p}$ to $\Sigma_{p^k}$ (under the obvious inclusion). Proposition 5.2 of \cite{subgroups} gives an isomorphism
\[
\sub_{k}(\G_{\E}) = \sub_{k}(\G_{\E}[p^k]),
\]
where $\G_{\E}[p^k]$ is the $p^k$-torsion of $\G_{\E}$.

Let $A$ be a finite abelian group. In Section 7 of \cite{subgroups}, Strickland constructs a formal scheme
\[
\level(A, \G_{\E}): \text{complete local Noetherian $\E^0$-algs} \lra{} \text{Set}
\]
that sends an $\E^0$-algebra $R$ to the level $A$-structures of $R \otimes \G_{\E}$. We recall this scheme because it will show up in the proof of Theorem \ref{thm2}.

Recall that there is a topological definition of $\G_{\E}[p^k]$:
\[
\Gamma (\G_{\E}[p^k]) = \E^0(B\Zp{k}).
\]
With a coordinate, by the Weierstrass preparation theorem, there are isomorphisms
\[
\Gamma(\G_{\E}[p^k]) \cong \E^0\powser{x}/[p^k]_{\G_{\E}}(x) \cong \E^0[x]/(f(x)),
\]
where $[p^k]_{\G_{\E}}(x)$ is the $p^k$-series of the formal group law and $f(x)$ is a monic polynomial of degree $p^{kn}$.

Because $\G_{\E}[p^k]$ is finite and free over $\Spf_{I_n}(\E^0)$ we may consider it over $\Spec(\E^0)$. Then it is a functor
\[
\G_{\E}[p^k]: \E^0 \text{-algebras} \lra{} \text{Abelian Groups}.
\]

Both of the formal schemes $\sub_k(\G_{\E})$ and $\level(A,\G_{\E})$ can be viewed as non-formal schemes as well without difficulty because they are finite and free over $\Spf(\E^0)$. We get 
\[
\sub_k(\G_{\E}): \E^0\text{-algebras} \lra{} \text{Set}
\]
sending an $\E^0$-algebra $R$ to the collection of subgroup schemes of order $p^k$ in $R\otimes \G_{\E}[p^k]$ (viewed as a non-formal scheme). By its definition the functor retains the property that
\[
R \otimes \sub_k(\G_{\E}) \cong \sub_k(R \otimes \G_{\E}).
\]

From now on we will write $\sub_k(\G_{\E})$ for the scheme over $\Spec(\E^0)$.

\subsection{Examples}
The goal of this section is to apply Theorem \ref{mainthm} to $\E^0(B\Sigma_{p^k})$ in some very particular examples in order to understand the effect of base change to $C_t$ on $\sub_k(\G_{\E})$.

A direct application of Theorem \ref{mainthm} provides a decomposition of $C_t \otimes \sub_k(\G_{\E})$ as a disjoint union of smaller schemes. Consider $\Sigma_{p^{k-1}}^{\times p} \subseteq \Sigma_{p^k}$. Theorem \ref{mainthm} gives the commutative square of rings
\[
\xymatrix{\E^0(B\Sigma_{p^{k-1}}^{\times p}) \ar[r] \ar[d]^{\Tr_{\E}} & \Prod{[\beta] \in \hom(\Z_{p}^{n-t}, \Sigma_{p^{k-1}}^{\times p})/\sim} C_{t}^{0}(BC(\im \beta)) \ar[d] \\
		\E^0(B\Sigma_{p^k}) \ar[r] & \Prod{[\al] \in \hom(\Z_{p}^{n-t}, \Sigma_{p^k})/\sim} C_{t}^{0}(BC(\im \al))}
\]
with the property that, after base change to $C_t$, there are isomorphisms
\[
\xymatrix{C_t\otimes_{\E^0}\E^0(B\Sigma_{p^{k-1}}^{\times p}) \ar[r]^-{\cong} \ar[d] & \Prod{[\beta] \in \hom(\Z_{p}^{n-t}, \Sigma_{p^{k-1}}^{\times p})/\sim} C_{t}^{0}(BC(\im \beta)) \ar[d] \\
		C_t\otimes_{\E^0}\E^0(B\Sigma_{p^k}) \ar[r]^-{\cong} & \Prod{[\al] \in \hom(\Z_{p}^{n-t}, \Sigma_{p^k})/\sim} C_{t}^{0}(BC(\im \al)).}
\]
By taking the quotient by the ideal generated by the image of the transfer we get the isomorphism
\begin{equation} \label{equa}
C_t \otimes_{\E^0} \E(B \Sigma_{p^k})/I_{tr} \cong \Prod{[\al] \in \hom(\Z_{p}^{n-t}, \Sigma_{p^k})/\sim} C_{t}^{0}(BC(\im \al))/I_{tr}^{[\al]},
\end{equation}
where
\[
I_{tr}^{[\al]} \subseteq C_{t}^{0}(BC(\im \al)).
\]
Theorem \ref{mainthm} allows us to compute $I_{tr}^{[\al]}$.
By Theorem \ref{strickland} the left hand side of isomorphism (\ref{equa}) is the global sections of 
\begin{eqnarray*}
C_t\otimes\sub_k(\G_{\E}) &\cong & C_t \otimes \sub_k(\G_{\E}[p^k]) \\
&\cong& \sub_k(C_t\otimes \G_{\E}[p^k]) \\
&\cong& \sub_k(\G_{C_t}[p^k]\oplus (\Zp{k})^{n-t}).
\end{eqnarray*}
The right hand side of isomorphism (\ref{equa}) is a product of $C_t$-algebras indexed by
\[
\hom(\Z_{p}^{n-t}, \Sigma_{p^k})/\sim.
\]
Of course, some of the $C_t$-algebras may be zero after taking the quotient by the images of the transfers. 

We apply Theorem \ref{mainthm} in some particular examples in order to study the phenomena described above. 
\begin{example} \label{firstexample}
The purpose of this example is to use Theorem \ref{mainthm} to compute the decomposition of $\sub_1(\G_{\E})$ after base change to $C_{n-1}$. Let $G = \Sigma_p$ and $H = e = \Sigma_{1}^{p}$. Then $H$ is the subgroup of $G$ that we use to define $I_{tr}$. There are precisely two conjugacy classes in
\[
\hom(\Z_p,\Sigma_p)
\]
corresponding to the trivial map and the map picking out the cyclic subgroup of order $p$. The centralizer of the image of the trivial map is $\Sigma_p$ and the centralizer of $\Z/p \subseteq \Sigma_p$ is just $\Z/p$. Thus the transchromatic generalized character map is an isomorphism
\[
C_{n-1}\otimes_{\E^0}\E^0(B\Sigma_p) \lra{\cong} C_{n-1}^{0}(B\Sigma_p) \times C_{n-1}^{0}(B\Z/p).
\]
Theorem \ref{mainthm} allows us to calculate the transfer
\[
C_{n-1}^{0} \lra{} C_{n-1}^{0}(B\Sigma_p) \times C_{n-1}^{0}(B\Z/p).
\]
The map to the first factor is a sum over $\Sigma_p/\Sigma_p \simeq *$ and Corollary \ref{maincor} gives the transfer from $e$ to $\Sigma_p$ for the cohomology theory $C_{n-1}$. The map on the second factor is a sum of transfers over 
\[
(G/H)^{\im \al}/C(\im \al) = (\Sigma_p/e)^{\Z/p}/\Z/p = \emptyset.
\]
Thus the map to the second factor is just the zero map. 

Let $\sub_1(\G_{C_{n-1}})$ be $\Spec(C_{n-1}^{0}(B\Sigma_p)/I_{tr}^{[e]})$, where $I_{tr}^{[e]}$ is the ideal generated by the image of the transfer from $e \subset \Sigma_p$.
We conclude that
\[
C_{n-1}\otimes \sub_1(\G_{\E}) \cong \sub_1(\G_{C_{n-1}}) \coprod \G_{C_{n-1}}[p].
\]

When $p=2$ it is easy to use a coordinate to calculate this map explicitly because $\Sigma_2 \cong \Z/2$. The isomorphism comes from the decomposition of $C_{n-1}\otimes \sub_1(\G_{\E})$ coming from the projection
\[
\G_{C_{n-1}} \oplus \Q_2/\Z_2 \lra{} \Q_2/\Z_2.
\]  
A subgroup of order $2$ can project onto $e \subset \Q_2/\Z_2$ or $\Z/2 \subset \Q_2/\Z_2$. If a subgroup projects onto $e$ then it is a subgroup of order two in $\sub_1(\G_{C_{n-1}})$. If the subgroup projects onto $\Z/2$ then every two torsion element $r \in \G_{C_{n-1}}$ defines a new subgroup of order two, the subgroup generated $(r,1)$.

For general $p$, the decomposition arises in the same way. The easiest way to see this is by considering the surjection
\[
\level(\Z/p, \G_{\E}) \lra{} \sub_1(\G_{\E}).
\]
This is how we proceed in the proof of the Theorem \ref{thm2}.

\end{example}

Before coming to the main theorem we work one more example.

\begin{example}
For this example let $p = 2$, and $t=n-1$. Let $G = \Sigma_4$ and $H = \Sigma_2 \times \Sigma_2$. Thus we are interested in understanding what topology has to say about the decomposition of 
\[
\sub_2(\G_{E_n})
\]
after base change to $C_{n-1}$.

There are precisely four conjugacy classes in
\[
\hom(\Z_2,\Sigma_4)
\]
corresponding to the cycle decompositions of $2$-power order elements. It is easy to check that
\begin{eqnarray*}
C(e) &\cong& \Sigma_4 \\
C((12)) &\cong& \Z/2\times \Z/2 \\
C((12)(34)) &\cong& D_8 \\
C((1234)) &\cong& \Z/4. 
\end{eqnarray*}
The transchromatic generalized character map is an isomorphism
\[
C_{n-1}\otimes_{E_{n}^{0}}E_{n}^{0}(B\Sigma_4) \cong C_{n-1}^{0}(B\Sigma_4) \times C_{n-1}^{0}(B\Z/2\times \Z/2) \times C_{n-1}^{0}(BD_8) \times C_{n-1}^{0}(B\Z/4).
\]
The transfer associated to $\Sigma_4$ is just the transfer from $\Sigma_2\times \Sigma_2$. The centralizer of $(12)$ in $H = \Sigma_2 \times \Sigma_2$ is $H$ and this implies that the transfer along $C_H((12)) \subseteq C_G((12))$ is the identity map. The centralizer $C_{\Sigma_2 \times \Sigma_2}((12)(34)) \subseteq \Sigma_2\times \Sigma_2$ is the whole group. Thus the transfer for $D_8$ is the transfer along $\Sigma_2 \times \Sigma_2 \subset C_{\Sigma_4}((12)(34))$. The transfer associated to $\Z/4$ is the zero map. 

Thus the scheme decomposes into the parts
\begin{equation} \label{decomp}
C_{n-1} \otimes \sub_2(\G_{E_n}) \cong \sub_2(\G_{C_{n-1}}) \coprod \G_{C_{n-1}}[4] \coprod X,
\end{equation}
where $X$ is the component (or components) corresponding to $D_8$.

Once again there is a natural decomposition of this sort from the algebraic geometry. 
The projection
\[
\G_{C_{n-1}} \oplus \Q_2/\Z_2 \lra{} \Q_2/\Z_2
\]
induces a map
\[
\sub_2(\G_{C_{n-1}}\oplus \Q_2/\Z_2) \lra{} \sub_{\leq 2}(\Q_2/\Z_2).
\]
The fibers of the points in the base consist of the subgroups that map to $e$, $\Z/2$, and $\Z/4$ in $\Q_2/\Z_2$.

The first two components in the decomposition (\ref{decomp}) seem to come from the subgroups that map onto $e$ and $\Z/4$ in $\Q_2/\Z_2$. Thus the third component must correspond to the subgroups that map to $\Z/2$ in $\Q_2/\Z_2$. Theorem \ref{thm2} implies that this is precisely the decomposition captured by the character map. That is, the scheme
\[
\Spec(C_{n-1}^{0}(BD_8)/I_{tr}^{[(12)(34)]})
\]
represents subgroup schemes of order four in $\G_{C_{n-1}} \oplus \Q_2/\Z_2$ that project onto $\Z/2 \subset \Q_2/\Z_2$.

\end{example}

\subsection{The decomposition}
Consider the projection
\[
C_t \otimes \G_{\E} \cong \G_{C_t} \oplus \QZ^{n-t} \lra{} \QZ^{n-t}.
\]
This induces a surjective map of schemes
\[
\sub_k(\G_{C_t} \oplus \QZ^{n-t}) \lra{} \sub_{\leq k}(\QZ^{n-t}).
\]
In this section we prove that the decomposition of
\[
\sub_k(\G_{C_t} \oplus \QZ^{n-t})
\]
as the disjoint union of the fibers of this map is a maximal decomposition and that the transchromatic generalized character map and Theorem \ref{mainthm} give precisely this decomposition.

\begin{lemma}
For any finite group $G$ the ring $C_{t}^{0}(BG)$ is connected.
\end{lemma}
\begin{proof}
Let $(C_{t})_{I_t}$ be the localization of $C_{t}$ at the prime ideal $I_t$. Let $K$ be the completion of $(C_{t})_{I_t}$ at the ideal $I_t$. The ring $K$ is a flat $C_{t}$-algebra because completions and localizations are flat, it is also complete local. Thus $K$ can be used to construct a new Borel-equivariant cohomology theory on finite $G$-spaces
\[
X \mapsto K \otimes_{C_t} C_{t}^{0}(EG \times_G X).
\]
The proof that $\E^0(BG)$ is complete local (eg. \cite{marshthesis}, Lemma 4.58 and Proposition 4.60) implies that $K \otimes_{C_t} C_{t}^{0}(BG)$ is complete local with respect to the ideal $I_t + I_{\text{aug}}$, where $I_{\text{aug}}$ is defined as in Proposition \ref{augmentation}. Now if $C_{t}^{0}(BG) \cong R_1 \times R_2$ for non-zero rings $R_1$ and $R_2$ then there is a split short exact sequence of $C_t$-modules (because $R_1$ and $R_2$ are necessarily $C_t$-algebras)
\[
0 \lra{} R_1 \lra{} R_1 \times R_2 \lra{} R_2 \lra{} 0.
\]
Tensoring up to $K$ preserves this sequence. However, $K \otimes_{C_t} C_{t}^{0}(BG)$ is connected.
\end{proof}

\begin{cor}
Let $H \subseteq G$ with $|G/H|$ divisible by $p$. Let $I_{tr} \subseteq C_{t}^{0}(BG)$ be the ideal generated by the image of the transfer from $H$ to $G$, then $C_{t}^{0}(BG)/I_{tr}$ is connected.
\end{cor}
\begin{proof}
If $|G/H|$ is not divisible by $p$ then the transfer map is surjective. Note that $I_{tr} \subseteq (p)+I_{\text{aug}}$ by Proposition \ref{augmentation}. There is a map of cohomology theories
\[
C_{t}^{0}(EG \times_G X) \lra{} K \otimes_{C_{t}} C_{t}^{0}(EG \times_G X),
\]
where $K$ is the $C_t$-algebra defined in the previous lemma. As transfer maps commuted with maps of cohomology theories we have
\[
\xymatrix{C_{t}^{0}(BH) \ar[d]^{\Tr} \ar[r] & K \otimes_{C_t} C_{t}^{0}(BH) \ar[d]^{\Tr} \\ C_{t}^{0}(BG) \ar[r] & K \otimes_{C_t} C_{t}^{0}(BG).}
\]
This implies that 
\[
K \otimes_{C_{t}} (C_{t}^{0}(BG)/I_{tr}) \cong (K \otimes_{C_t} C_{t}^{0}(BG))/I_{tr},
\]
where the ideal $I_{tr}$ on the left is the one defined using the left arrow and the ideal on the right is defined using the right arrow. This ring is local (and thus connected). The argument from the previous lemma now implies the claim.
\end{proof}

Recall that the transchromatic generalized character map and Theorem \ref{mainthm} give an isomorphism
\[
C_t \otimes_{\E^0} \E^0(B \Sigma_{p^k})/I_{tr} \cong \Prod{[\al] \in \hom(\Z_{p}^{n-t}, \Sigma_{p^k})/\sim} C_{t}^{0}(BC(\im \al))/I_{tr}^{[\al]}
\]
in which the ideal $I_{tr}$ on the left is the ideal generated by the image of the transfer $\Sigma_{p^{k-1}}^{\times p} \subset \Sigma_{p^k}$ and the ideals called $I_{tr}^{[\al]}$ on the right are determined by Theorem \ref{mainthm}.

The following is our main combinatorial result.

\begin{lemma} \label{combinatorial}
The number of non-zero factors of 
\[
\Prod{[\al] \in \hom(\Z_{p}^{n-t}, \Sigma_{p^k})/\sim} C_{t}^{0}(BC(\im \al))/I_{tr}^{[\al]}
\]
are in bijective correspondence with the elements of 
\[
\sub_{\leq k}(\QZ^{n-t}).
\]
\end{lemma}
\begin{proof}
This is a question about when the transfer map is surjective. It is true that some of the ideals $I_{tr}$ in the statement of the lemma are generated by the image of two or more transfer maps. However, since these ideals are contained in $(p)+I_{\text{aug}} \subset C_{t}^{0}(BG)$ (unless the ideal is the whole ring), $I_{tr}$ is the whole ring if and only if one of the transfer maps is surjective. 

Let $h = n-t$. It is well-known (see Section 3 of \cite{etheorysym}, for instance) that the number of conjugacy classes of maps
\[
\Z_{p}^{h} \lra{} \Sigma_{p^k}
\]
that do not lift (up to conjugacy) to
\[
\Sigma_{p^{k-1}}^{\times p} \subseteq \Sigma_{p^k}
\]
is in bijective correspondence with isomorphism classes of transitive $\Z_{p}^{h}$-sets and this is in bijective correspondence with 
\[
\sub_{k}(\QZ^{h}).
\]
It is clear that all maps with this property contribute factors to the product in question: the transfer map is the zero map. 

Now fix a map $\al: \Z_{p}^{h} \lra{} \Sigma_{p^k}$ that does factor (up to conjugacy) through
\[
\xymatrix{& \Sigma_{p^{k-1}}^{\times p} \ar[d]^i \\ \Z_{p}^{h} \ar[r]^{\al} \ar@{-->}[ur] & \Sigma_{p^k}}
\]
and let $\gamma_1, \ldots, \gamma_l$ represent elements of $i_{*}^{-1}([\al])$.

Let $m < k$ be the smallest integer such that a map $\al:\Z_{p}^{h} \lra{} \Sigma_{p^k}$ factors up to conjugacy through
\[
\Sigma_{p^m}^{\times p^{k-m}} \subseteq \Sigma_{p^k}.
\]
Since $\gamma_1, \ldots, \gamma_l$ all represent isomorphic $\Z_{p}^{h}$-sets (because they are all conjugate in $\Sigma_{p^k}$) the integer $m$ is also the smallest integer such that, for each $i$, there is a factorization
\[
\xymatrix{& \Sigma_{p^{m}}^{\times p^{k-m}} \ar[d] \\ \Z_{p}^{h} \ar[r]^{\gamma_i} \ar@{-->}[ur] & \Sigma_{p^{k-1}}^{\times p}}
\]
up to conjugacy in $\Sigma_{p^{k-1}}^{\times p}$.

Now assume that $\al$ does not factor through the diagonal map
\[
\Sigma_{p^m} \lra{\triangle} \Sigma_{p^m}^{\times p^{k-m}}.
\]
We show that, in this case, the transfer from $C_{\Sigma_{p^{k-1}}^{\times p}}(\im \al) \lra{} C_{\Sigma_{p^k}}(\im \al)$ is the identity map. 

Let $X$ be the $\Z_{p}^{h}$-set associated to $\al$. The factorization determines $p^{k-m}$ $\Z_{p}^{h}$-sets of order $p^m$: $X_1, \ldots, X_{p^{k-m}}$ such that $X \cong X_1 \coprod \ldots \coprod X_{p^{k-m}}$. 

The fact that $m$ is the smallest integer with this property implies that at least one of the $\Z_{p}^{h}$-sets of order $p^m$ is transitive. Without loss of generality we may assume that $X_1$ is transitive and that $X_1, X_2, \ldots X_j$ are isomorphic $\Z_{p}^{h}$-sets and $X_{j+1}, \ldots, X_{p^{k-m}}$ are all non-isomorphic to $X_1$. 

Note that $j$ may be equal to $1$ and that we know there are non-isomorphic $\Z_{p}^{h}$-sets because the map $\al$ does not factor through the diagonal.

By \cite{etheorysym} Lemma 8.11 it suffices to show that the transfer from 
\[
C_{\Sigma_{p^{jm}} \times \Sigma_{p^{k-jm}}}(\im \al) \subset C_{\Sigma_{p^k}}(\im \al)
\]
is the identity.

Now consider an element $\sigma \in C_{\Sigma_{p^k}}(\im \al)$, this determines an automorphism of $X$. Since $X_1, \ldots, X_j$ are transitive and not isomorphic to the other $\Z_{p}^{h}$-sets, $\sigma$ can not map any of $X_1, \ldots, X_j$ to $X_k$ with $k > j$. Thus $\sigma$ must be the product of two disjoint permutations. In other words $\sigma \in \Sigma_{p^{jm}} \times \Sigma_{p^{k-jm}}$ and this implies that the transfer map described above is induced by the identity map on groups.

Next assume that $\al$ factors (up to conjugacy) through the $\Delta$:
\[
\xymatrix{& \Sigma_{p^m} \ar[d]^{\triangle} \\ & \Sigma_{p^m}^{\times p^{k-m}} \ar[d] \\ \Z_{p}^{h} \ar[r]^{\al} \ar@{-->}[uur] & \Sigma_{p^k}.}
\]
This implies that each of the $\gamma_i$'s will factor through the diagonal (up to conjugacy in $\Sigma_{p^{k-1}}^{\times p}$). We also know that $\al$ does not factor through the inclusion $\Sigma_{p^{m-1}}^{\times p} \subset \Sigma_{p^m}$. We will conclude that the transfer map induced by the inclusion
\[
C_{\Sigma_{p^{k-1}}^{\times p}}(\im \al) \lra{} C_{\Sigma_{p^k}}(\im \al)
\]
is not the identity map.

The assumptions imply that the dotted arrow determines a transitive $\Z_{p}^{h}$-set of order $p^m$ and that $X$ is a disjoint union of $p^{k-m}$ copies of this set. Now any permutation of these sets is in $C_{\Sigma_{p^k}}(\im \al)$ and many of these are elements of prime power order that are not in $\Sigma_{p^{k-1}}^{\times p}$. 

Now any map $\al:\Z_{p}^{h} \lra{} \Sigma_{p^k}$ factors up to conjugacy through one of the two cases discussed above. In the first case, when it does not factor through the diagonal, it does not contribute a factor to the product in question (because $I_{tr}^{[\al]} = C_{t}^{0}(BC(\im \al))$). In the second case, when it does factor through the diagonal, then it does contribute a factor. In this case the number of maps $\al$ (up to conjugacy) with a particular $m$ are in bijective correspondence with the number of isomorphism classes of $\Z_{p}^{h}$-sets of order $p^m$. This is the cardinality of $\sub_{m}(\QZ^{h})$. Putting these together for varying $m$ gives the total number of nontrivial factors in the product: the cardinality of 
\[
\sub_{\leq k}(\QZ^{n-t}).
\]
\end{proof}

The isomorphism induced by the character map
\[
\Coprod{[\al] \in \hom(\Z_{p}^{n-t}, \Sigma_{p^k})/\sim} \Spec(C_{t}^{0}(BC(\im \al))/I_{tr}^{[\al]}) \cong \sub_{k}(\G_{C_t} \oplus \QZ^{n-t})
\]
along with the lemmas and examples above seem to imply that the character map modulo transfers witnesses the decomposition of $\sub_{k}(\G_{C_t} \oplus \QZ^{n-t})$ as the fibers of the map to $\sub_{\leq k}(\QZ^{n-t})$. The first lemma implies that the scheme can be decomposed no further. Now we show that this is true:

\begin{thm} \label{thm2}
The isomorphism fits into a commutative triangle
\[
\xymatrix{\Coprod{[\al] \in \hom(\Z_{p}^{n-t}, \Sigma_{p^k})/\sim}\Spec(C_{t}^{0}(BC(\im \al))/I_{tr}^{[\al]}) \ar[r]^-{\cong} \ar[d] & \sub_{k}(\G_{C_t} \oplus \QZ^{n-t}) \ar[dl] \\ \sub_{\leq k}(\QZ^{n-t}), &}
\]
where the left map takes the component corresponding to $[\al]$ to the image of 
\[
\al^*: (\im \al)^* \lra{} \QZ^{n-t}
\]
and the right map is induced by the projection
\[
\G_{C_t} \oplus \QZ^{n-t} \lra{} \QZ^{n-t}.
\]
\end{thm}
\begin{proof}
Note that the image of the Pontryagin dual in $\QZ^{n-t}$ is invariant under conjugation of the map $\al$. The right vertical map is induced by projection onto the $\QZ^{n-t}$ factor.

Let $A$ be an abelian group of order $p^k$. There is a canonical isomorphism
\[
\hom(A^*, \G_{\E}) \cong \Spec(\E^0(BA)).
\]
Pulling this isomorphism back to the ring $C_t$ and applying the transchromatric generalized character map gives the isomorphism
\[
\hom(A^*, \G_{C_t}\oplus \QZ^{n-t}) \cong \Coprod{\hom(\Z_{p}^{n-t},A)}\Spec(C_{t}^{0}(BA)).
\]
The definition of the character map implies that this fits into the following commutative diagram:
\[
\xymatrix{\hom(A^*, \G_{C_t}\oplus \QZ^{n-t}) \ar[r]^-{\cong} \ar[d] &  \Coprod{\hom(\Z_{p}^{n-t},A)}\Spec(C_{t}^{0}(BA)) \ar[d] \\
			\hom(A^*,\QZ^{n-t}) \ar[r]^{(-)^*} \ar[d] & \hom(\Z_{p}^{n-t}, A) \ar[ld] \\
			\sub_{\leq k}(\QZ^{n-t}). &}
\]
There is also the commutative diagram of schemes
\[
\xymatrix{\level(A^*, \G_{\E}) \ar[d] \ar[r]^-{\cong} & \Spec(\E^0(BA)/I_{tr}) \ar[d] \\
			\hom(A^*,\G_{\E}) \ar[r]^-{\cong} & \Spec(\E^0(BA)).}
\]
Pulling the top arrow back to the ring $C_t$ and then applying the character map gives the commutative diagram
\[
\xymatrix{\level(A^*, \G_{C_t} \oplus \QZ^{n-t}) \ar[r]^-{\cong} \ar[d] & \Coprod{\al \in \hom(\Z_{p}^{n-t},A)} \Spec(C_{t}^{0}(BA)/I_{tr}^{\al}) \ar[ld] \\
			\sub_{\leq k}(\QZ^{n-t}). &}
\]
It should be noted that a level structure for the \pdiv group $\G_{C_t} \oplus \QZ^{n-t}$ is a map $A^* \lra{} \G_{C_t} \oplus \QZ^{n-t}$ that is either a level structure for $\G_{C_t}$ or injective into $\QZ^{n-t}$. It follows immediately from Theorem \ref{mainthm} that this is what the top right corner of the diagram represents.

In \cite{subgroups}, Theorem 7.4, Strickland defines a map
\[
\level(A^*, \G_{\E}) \lra{} \sub_{k}(\G_{\E}).
\]
The map sends a level structure to its ``image'', which is a subgroup scheme of $\G_{\E}$. When the target scheme is constant the ``image'' divisor of the level structure is the genuine image of the map. Thus after pulling back to $C_t$ we have the following commutative diagram:
\[
\xymatrix{\level(A^*, \G_{C_t} \oplus \QZ^{n-t}) \ar[r] \ar[d] & \sub_{k}(\G_{C_t} \oplus \QZ^{n-t}) \ar[ld] \\
			\sub_{\leq k}(\QZ^{n-t}). &}
\]
In the proof of Proposition 9.1 of \cite{etheorysym}, Strickland proves the following result: Let $\bar{A}$ be the set of transitive abelian subgroups of $\Sigma_{p^k}$ (note that each of these has order $p^k$). The following diagram commutes:
\[
\xymatrix{\Coprod{A \in \bar A}\level(A^*,\G_{\E}) \ar[r]^-{\cong} \ar[d] & \Coprod{A \in \bar A} \Spec(\E^0(BA)/I_{tr}) \ar[d] \\
			\sub_{k}(\G_{\E}) \ar[r]^-{\cong} & \Spec(\E^0(B\Sigma_{p^k})/I_{tr}),}
\]
where the right map is induced by the inclusion $A \subseteq \Sigma_{p^k}$ and the global sections of each of the vertical maps are injective maps of rings. Note that this property is preserved after pull-back to $C_t$ because $C_t$ is a flat $\E^0$-algebra.

We have shown that, after pulling back to $C_t$, the left hand map and the top map both commute with the natural maps to $\sub_{\leq k}(\QZ^{n-t})$. Because the right hand map is induced (on each component) by an inclusion of groups, the subgroups of $\QZ^{n-t}$ defined by considering the image of the Pontryagin dual of the map from $\Z_{p}^{n-t} \lra{} \im \al \subseteq A$ or $\Z_{p}^{n-t} \lra{} \im \al \subseteq A \subseteq \Sigma_{p^k}$ are the same. This implies that the right hand arrow also sits inside a commutative triangle to $\sub_{\leq k}(\QZ^{n-t})$. 

Finally, since the global sections of the vertical maps are injective we can pick an element in the global sections of $\sub_{\leq k}(\QZ^{n-t})$, map it into the global sections of $\sub_k (\G_{C_t} \oplus \QZ^{n-t})$ and then map it around the square. The result follows. 
\end{proof}

Fix a map $\al: \Z_{p}^{n-t} \lra{} \Sigma_{p^k}$ that factors through $\Delta$ (up to conjugacy) as in the proof of Lemma \ref{combinatorial} and let $L \subseteq \QZ^{n-t}$ be the image of the Pontryagin dual $\al^*: \im \al \lra{} \QZ^{n-t}$. Let $f:\sub_{k}(\G_{C_t} \oplus \QZ^{n-t}) \lra{} \sub_{\leq k}(\QZ^{n-t})$ and let $f^{-1}(L)$ be the pullback 
\[
\xymatrix{f^{-1}(L) \ar[r] \ar[d] & \sub_{k}(\G_{C_t} \oplus \QZ^{n-t}) \ar[d]^{f} \\
			\ast \ar[r]^-{L} & \sub_{\leq k}(\QZ^{n-t}).}
\]
We have the following corollary of Theorem \ref{thm2} above that gives an algebro-geometric description of the $C_t$ cohomology of groups that arise as centralizers of tuples of commuting elements in symmetric groups (modulo a transfer ideal):
\begin{cor}
For $\al:\Z_{p}^{n-t} \lra{} \Sigma_{p^k}$ factoring (up to conjugacy) through $\Delta$, there is an isomorphism 
\[
\Spec(C_{t}^{0}(BC(\im \al))/I_{tr}^{[\al]}) \cong f^{-1}(L),
\]
where the ideal $I_{tr}^{[\al]}$ is the ideal coming from the application of Theorem \ref{mainthm} to the inclusion $\Sigma_{p^{k-1}}^{\times p} \subset \Sigma_{p^k}$.
\end{cor}
\begin{proof}
This follows immediately from the previous theorem.
\end{proof}

\begin{remark}
When $t = n-1$ the groups that arise as centralizers of maps $\al:\Z_p \lra{} \Sigma_{p^k}$ that factor through $\Delta$ are groups of the form
\[
\Zp{i} \wr \Sigma_{p^j},
\] 
where $i+j = k$.
\end{remark}

\begin{remark}
When $\im \al = e \subset \Sigma_{p^k}$, the fiber over $e \in \sub_{\leq k}(\QZ^{n-t})$ is $\sub_k(\G_{C_t})$. 
\end{remark}

\begin{remark}
When $\im \al = \Z/p^k \subset \Sigma_{p^k}$, the image of the Pontryagin dual is a subgroup of $\QZ^{n-t}$ isomorphic to $\Z/p^k$. The fiber is $\G_{C_t}[p^k]$. 
\end{remark}

Before this, the two classes of finite groups with algebro-geometric interpretations of their cohomology rings were cyclic groups and symmetric groups. The remarks above imply that the fibers of the subgroups between $e$ and $\Zp{k}$ can be viewed as interpolating between these two examples.

\bibliographystyle{abbrv}
\bibliography{mybib}

\end{document}